\documentclass[11pt]{article}

\date{14th August 2013}
\def\titlename{ZL-amenability constants of finite groups with two character degrees}
\def\authname{M. Alaghmandan, Y. Choi, and E. Samei}

\title{\vskip-1.0em{\sc \titlename}}
\author{{\sc \authname}}

\usepackage{amsmath,amssymb}
\newcounter{pulse}[section]
\numberwithin{pulse}{section}
\numberwithin{equation}{section}

\usepackage{amsthm}
\newcommand{\tf}{\sc}

\newtheorem{thm}[pulse]{\tf Theorem}
\newtheorem{prop}[pulse]{\tf Proposition}
\newtheorem{lem}[pulse]{\tf Lemma}

\theoremstyle{definition}
\newtheorem{dfn}[pulse]{\tf Definition}
\newtheorem{eg}[pulse]{\tf Example}
\newtheorem{rem}[pulse]{\tf Remark}
\newtheorem*{qu}{\tf Question}

\theoremstyle{plain}
\newtheorem*{propstar}{\tf Proposition}

\newenvironment{numlist}{%
\begin{enumerate}

}{\end{enumerate}\ignorespacesafterend}

\newcommand{\dt}[1]{{\it#1}\/}  
\newcommand{\st}{\colon}
\newcommand{\defeq}{:=}  

\newcommand{\Cplx}{{\mathbb C}} 

\newcommand{\bbF}{{\mathbb F}}

\newcommand{\fL}{{\mathfrak L}}

\newcommand{\Abs}[1]{\left\vert #1 \right\vert}

\newcommand{\abs}[1]{\vert #1 \vert}
\newcommand{\norm}[1]{\Vert #1 \Vert}

\newcommand{\ptp}{\mathbin{\widehat{\otimes}}}

\newcommand{\Irr}{\operatorname{Irr}} 
\newcommand{\der}[1]{{#1}'}  
\newcommand{\Gab}{{G_{\rm ab}}}  
\newcommand{\conj}{\operatorname{Conj}}

\newcommand{\Aff}{\operatorname{Aff}}  
\newcommand{\affine}[2]{\left(\begin{matrix} #1 & #2 \\ 0 & 1 \end{matrix}\right)}

\newcommand{\Zl}{\operatorname{Z\ell}}
\newcommand{\AM}{\operatorname{AM}}
\newcommand{\AMZL}[1]{\AM_{\rm Z}(#1)} 

\usepackage[usenames]{color}
\renewcommand{\dt}[1]{\textcolor{Bittersweet}{\sf#1}}
\usepackage{hyperref}

\usepackage{sectsty}
\allsectionsfont{\sc}

\begin{document}

\maketitle

\begin{abstract}
We calculate the exact amenability constant of the centre of $\ell^1(G)$ when $G$ is one of the following classes of finite group: dihedral; extraspecial; or Frobenius with abelian complement and kernel. This is done using a formula which applies to all finite groups with two character degrees. In passing, we answer in the negative a question raised in work of the third author with Azimifard and Spronk (J.~Funct. Anal.~2009).

\bigskip
\noindent{\it Keywords}: Center of group algebras, characters, character degrees, amenability constant, Frobenius group, extraspecial groups.

\smallskip
\noindent{\it MSC 2010}: 43A20 (primary); 20C15 (secondary).
\end{abstract}

\begin{section}{Introduction}
Let $G$ be a compact group. While its group algebra $L^1(G)$ is always \dt{amenable} as a Banach algebra, the \emph{centre} of $L^1(G)$ need not be amenable.\footnote{It is not clear where credit belongs for this observation, but it is mentioned in \cite{Stegmeir} as a personal communication from B.~E. Johnson.}
The third author of the present paper, in work with A. Azimifard and N. Spronk \cite{AzSaSp}, examined this phenomenon in various cases, and showed that if $G$ is an infinite product of non-abelian finite groups, given the product topology, then the centre of $L^1(G)$ fails to be amenable. (Theorem 1.10, {\it op cit.}) The proof works by examining the so-called \dt{amenability constant} of the centre of $L^1(G)$ in the case where $G$ is \emph{finite}, and observing that this constant is nothing but the $\ell^1$-norm of a certain central idempotent in the complex group ring of $G\times G^{op}$. It then follows from results of D. Rider that the infimum of the possible amenability constants as $G$ varies over all finite, \emph{non-abelian} groups is \emph{strictly} greater than $1$; however, the lower bound provided by Rider's theorem seems far from best possible.

Motivated by this and other questions in \cite{AzSaSp}, we investigate the amenability constant of the centre of $L^1(G)$ in the case where $G$ is a finite group with \dt{two character degrees} (that is, nonabelian finite groups where all \emph{irreducible} non-linear characters have the same degree). For this special class of groups, we obtain an exact formula for the amenability constant, that is easily computed for various natural classes of such groups, including dihedral groups, extraspecial $p$-groups, and certain Frobenius groups. It will be seen that in some cases interesting patterns emerge, which remain to be fully explored.

\paragraph{Notation and terminology}
Given a group $G$ (with its discrete topology) we denote its $\ell^1$-group algebra by $\ell^1(G)$ and denote the centre of $\ell^1(G)$ by $\Zl^1(G)$. In this paper we are only concerned with finite groups.
We assume the reader has basic knowledge of finite group theory. For all other terminology and definitions pertaining to character theory, the book \cite{JamLie} will suffice, as should many other choices. See also Chapters 3 and 4 of Gorenstein's book~\cite{Gor_FGbook_ed2}.

Given a finite group $G$, we shall define the amenability constant of $\Zl^1(G)$ in Section~\ref{ss:general}. This constant will be denoted throughout by $\AMZL{G}$, and to save cumbersome repetition in verbal descriptions, we will refer to $\AMZL{G}$ as the \dt{ZL-amenability constant of $G$}. Although this is, to our knowledge, not standard terminology, it emphasizes that $\AMZL{G}$ can be computed from the character table of $G$ without any reference to Banach algebras and the general theory of amenability.
Indeed, one can take Equation~\eqref{eq:AMZL-formula} below as the \emph{definition} of $\AMZL{G}$ when $G$ is finite.
\end{section}

\begin{section}{A formula for the amenability constant}
\label{s:AMZL-formulas}

\begin{subsection}{General background}\label{ss:general}
We start with some general remarks on amenability constants. This material is not original: it is well known to specialists, and is either folklore or can be extracted from parts of~\cite{AzSaSp}. Nevertheless we include these remarks to provide context for the formulas which will follow. We will not discuss amenability for general Banach algebras here, and instead focus on the finite-dimensional setting.

If $A$ is a finite-dimensional Banach algebra with identity, then $A$ is \dt{amenable} if and only if there exists some $M\in A\otimes A$ satisfying $a\cdot M = M\cdot a$ and $\pi(M)=1_A$. Such an $M$ is called a \dt{diagonal element} for the algebra~$A$. (Thus far we have not used the norm on $A$. Indeed, the characterization of amenability that we have given for finite-dimensional algebras is equivalent to \dt{separability} in the sense of ring-theory; what we call a diagonal element is usually called a \dt{separability idempotent} in that context.)

\begin{eg}\label{eg:fingp-is-amenable}
Let $H$ be a finite group. Then it is straightforward to check that
\[ N\defeq |H|^{-1}\sum_{h\in H} \delta_h \otimes \delta_{h^{-1}}\]
is a diagonal element for the algebra $\ell^1(H)$, and that $\norm{N}=1$ as an element of the projective tensor product $\ell^1(H)\ptp\ell^1(H)=\ell^1(H\times H)$.
\end{eg}

If $A$ is a finite-dimensional, semisimple and commutative Banach algebra, then it is isomorphic as an algebra to $\Cplx^n$ for some $n$. Since $\Cplx^n$ has a unique diagonal element, so does $A$. More precisely, if the minimal idempotents of $A$ are $e_1,\dots, e_n$, then the unique diagonal element for $A$ is just $M=\sum_{j=1}^n e_j \otimes e_j$. In this setting the \dt{amenability constant of $A$} is the norm of $M$ as an element of the projective tensor product $A\ptp A$.

Now let $G$ be a finite group.
By basic results from representation theory, the centre of $\Cplx G$ is commutative and semisimple, and its minimal idempotents are the functions of the form $|G|^{-1}d_\chi \chi$ where $\chi$ is an irreducible character on $G$.
(See also \cite[Corollary 15.4]{JamLie} for an alternative approach to this result.)
Therefore, letting $\Irr(G)$ denote the set of irreducible characters on $G$, the
unique diagonal element of $\Zl^1(G)$ is
\begin{equation}\label{eq:diagonal}
M \defeq \sum_{\chi\in\Irr(G)}  \frac{d_\chi}{|G|} \chi \otimes \frac{d_\chi}{|G|}\chi
\end{equation}
We remark that this observation is also made\footnotemark\ in the proof of \cite[Theorem 1.8]{AzSaSp}.
\footnotetext{The reader should beware that in \cite[Section 1.5]{AzSaSp}, $G$ is equipped with normalized counting measure. Here we have chosen to equip $G$ with counting measure, since this ensures that $\Zl^1(G)$ has an identity element of norm~$1$.}

Let $\conj(G)$ denote the set of conjugacy classes in $G$, and let $\chi(C)$ denotes the value taken by a character $\chi$ on any (hence every) element of a conjugacy class $C\in\conj(G)$. Then, as shown in the proof of \cite[Theorem 1.8]{AzSaSp}, we find that
\begin{equation}\label{eq:AMZL-formula}
\AMZL{G}= \norm{M} = \frac{1}{|G|^2} \sum_{C,D\in\conj(G)} \left\vert \sum_{\chi \in \Irr(G)} d_\chi^2 \chi(C) \overline{\chi(D)}\right\vert\ |C| \ |D|.
\end{equation}
Equation~\eqref{eq:AMZL-formula} is difficult to work with if one wants to calculate the \emph{exact} ZL-amenability constant of~$G$, as it seems to require knowledge of the full character table of $G$. When $G$ has two character degrees, things simplify greatly.
\end{subsection}

\begin{subsection}{The case of two character degrees}
Following standard terminology, we say that a character on a finite group is \dt{linear} if it is actually a homomorphism $G\to {\mathbb T}$ (i.e. the trace of a $1$-dimensional representation) and \dt{non-linear} otherwise.

\begin{dfn}
We say that a finite group $G$ has \dt{two character degrees} if (i) $G$ is non-abelian (ii) every non-linear irreducible character has the same degree.
\end{dfn}

\begin{rem}
Groups with two character degrees were studied, in an unrelated context, by Isaacs and Passman \cite{IP_PJM68}. (In the terminology of that paper, the groups we consider \dt{have a.c.~$m$} for some integer $m>1$.)
It is known that such groups must be metabelian \cite[Corollary 12.6]{Isaacs_CTbook},
and they seem a natural class of examples to consider, if looking for large non-abelian groups with small ZL-amenability constants. We shall see some evidence to support this approach in Examples \ref{eg:AMZL-of-affine} and~\ref{eg:AMZL-of-extraspecial}.
\end{rem}

Given a finite group $G$, we denote by $\der{G}$ its \dt{derived subgroup} (also called its \dt{commutator subgroup}): this can be constructed as the normal subgroup generated by all commutators in $G$, or characterized as the smallest normal subgroup $N$ for which the quotient $G/N$ is abelian.

\begin{thm}\label{t:amzl-2cd}
Let $G$ be a finite group wth two character degrees, and let $m$ be the degree of any (hence every) non-linear irreducible group character of~$G$.
Then
\begin{equation}\label{eq:new-AMZL-formula}
\AMZL{G}= 1 + 2(m^2 -1) \left(1  - \frac{1}{|G|\ |\der{G}|} \sum_{C\in\conj(G)} |C|^2\right).
\end{equation}
\end{thm}

The advantage of \eqref{eq:new-AMZL-formula} over \eqref{eq:AMZL-formula} is that, once we know that every non-linear character of $G$ has dimension $m$, we only need two other pieces of information: the order of the derived subgroup (see Remark~\ref{r:counting-lin-char} below) and the size of each conjugacy class. In particular, we do not need the full character table of~$G$.

\begin{rem}\label{r:counting-lin-char}
Given an arbitrary finite group $G$, let $\fL$ denote the set of linear characters on $G$. Then $\fL$ is in bijection, in a natural way, with the set of characters on the abelian group $\Gab \defeq G/\der{G}$\/
(see for instance, \cite[Theorem 17.11]{JamLie}; and consequently $|\fL| = | \widehat{\Gab} | = | \Gab | = |G|\ /\ |\der{G}|$.
We shall use this fact in the proof of Theorem~\ref{t:amzl-2cd}. For some of our examples, it will be more convenient to count the number of linear characters than to work out the derived subgroup.
\end{rem}

Before proving Theorem~\ref{t:amzl-2cd}, we isolate one of the steps as a separate lemma.

\begin{lem}\label{l:ES_trick}
Let $G$ be an arbitrary finite group, and let $\fL=\{ \chi\in\Irr(G) \st d_\chi=1\}$. Then
\begin{equation}\label{eq:FRED}
\frac{1}{|G|^2} \sum_{C,D \in \conj(G)}  |C| \ |D|  \ \Abs{  \sum_{\chi\in \fL} \chi(C)\overline{\chi(D)} } = 1\,.
\end{equation}
\end{lem}

\begin{proof}
Let $\Gab$ denote the quotient group $G/\der{G}$ and let $q: G\to \Gab$ be the quotient homomorphism.
The left-hand side of \eqref{eq:FRED} is the norm, in $\Zl^1(G\times G)$, of the following idempotent:
\[ \widetilde{M} \defeq \frac{1}{|G|^2} \sum_{\chi\in\fL} \chi\otimes\chi \,.\]
Since each $\chi$ is constant on cosets of the derived subgroup,
 $\widetilde{M}$ factors through the quotient map $q\otimes q: \Zl^1(G\times G) \to \Zl^1(\Gab\times \Gab)=\ell^1(\Gab\times\Gab)$.
For each $x,y\in G$, put
\[ M(q(x),q(y)) \defeq | \der{G} |^2 \widetilde{M}(x,y);\]
then $M$ is a well-defined element of $\Zl^1(\Gab\times\Gab)$, and a little thought shows that $\norm{M}= \norm{\widetilde{M}}$.
On the other hand, since $\chi\in\fL$ if and only if $\chi= \phi\circ q$ for some $\phi\in \widehat{\Gab}$ (see Remark~\ref{r:counting-lin-char}),
\[ \begin{aligned}
M  = \frac{1}{\abs{\Gab}^2} \sum_{\phi\in \widehat{\Gab}} \phi\otimes \overline{\phi}
\end{aligned} \]
Comparing this with \eqref{eq:diagonal}, we see that
$M$ is the unique diagonal element for the Banach algebra $\Zl^1(\Gab) = \ell^1(\Gab)$. On the other hand, by Example~\ref{eg:fingp-is-amenable}, $\ell^1(\Gab)$ has a diagonal element of norm~$1$. It follows that $M$ has norm $1$, which completes the proof.
\end{proof}

\begin{proof}[Proof of Theorem~\ref{t:amzl-2cd}]
To ease notation, we write $\conj$ instead of $\conj(G)$ throughout this proof. Let
\[ \AM_{\rm diag} =  \frac{1}{|G|^2} \sum_{C\in\conj} |C|^2 \ \sum_{\chi\in\Irr(G)} d_\chi^2 \abs{\chi(C)}^2  \]
and
\[ \AM_{\rm off} = \frac{1}{|G|^2} \sum_{(C,D)\in\conj^2 \st C\neq D} |C|\ |D|\ \Abs{ \sum_{\chi\in\Irr(G)} d_\chi^2 \chi(C)\overline{\chi(D)} }  \]
so that, by \eqref{eq:AMZL-formula}, $\AMZL{G} = \AM_{\rm diag} + \AM_{\rm off}$.

Let $\fL=\{ \chi\in\Irr(G) \st d_\chi = 1\}$.
Using Schur column orthogonality,
 and the fact that $\abs{\chi(\cdot)}^2=1$ for every $\chi\in\fL$, we get
\begin{equation}\label{eq:diag-term}
\begin{aligned}
|G|^2 \AM_{\rm diag}
 &  = \sum_{C\in\conj} |C|^2 \left(
	 \sum_{\chi\in\Irr(G)} m^2 \abs{\chi(C)}^2 - \sum_{\chi\in \fL} (m^2-1)\abs{\chi(C)}^2
	 \right) \\
 & = \sum_{C\in\conj} |C|^2 \left( m^2 \frac{|G|}{|C|} - (m^2-1) |\fL| \right) \\
 & = m^2 |G|^2 - (m^2-1)\ |\fL| \sum_{C\in\conj} |C|^2\,. \\
\end{aligned}
\end{equation}

Similarly,
\begin{equation}\label{eq:off-diag-term}
\begin{aligned}
|G|^2 \AM_{\rm off}
 & = \sum_{(C,D)\st C\neq D} |C|\ |D|\ \Abs{
	\sum_{\chi\in\Irr(G)} m^2 \chi(C)\overline{\chi(D)} - \sum_{\chi\in \fL} (m^2-1)\chi(C)\overline{\chi(D)}
	} \\
 & = \sum_{(C,D)\st C\neq D} |C|\ |D|\ \Abs{
	\sum_{\chi\in \fL} (m^2-1)  \chi(C)\overline{\chi(D)}
	} \\
	& \qquad\qquad(\text{by Schur column orthogonality}) \\
 & =  (m^2-1)  \sum_{(C,D)} |C|\ |D|\ \Abs{
	\sum_{\chi\in\fL}\chi(C)\overline{\chi(D)}
 	} - (m^2-1)  \sum_{C} |C|^2 \ |\fL| \\
	 & \qquad\qquad(\text{since $\abs{\chi(\cdot)}^2=1$ for all $\chi\in \fL$}) \\
 & = (m^2-1)|G|^2 -  (m^2-1)  \sum_{C} |C|^2 \ |\fL|\quad,
\end{aligned}
\end{equation}
with the last equation following from Lemma~\ref{l:ES_trick}. Combining \eqref{eq:diag-term} and \eqref{eq:off-diag-term}, using
the equality $|\fL| = |G|\ /  |\der{G}|$, and rearranging terms, we obtain the desired formula.
\end{proof}

\end{subsection}

\subsection{Motivation: lower bounds on ZL-amenability constants}
As mentioned in the introduction, it is observed in \cite[\S1.5]{AzSaSp} that
\begin{equation}\label{eq:AMZL-gap}
\inf \{ \AMZL{G} \colon \text{$G$ finite and non-abelian} \} > 1.
\end{equation}
(See the proof of Theorem 1.10, {\it op.~cit.}\/) The proof relies on the following hard result of D. Rider.

\begin{thm}[Rider; {see \cite[Lemma 5.2]{ri}}]\label{t:rider}
Let $G$ be a compact group, $\lambda$ a Haar measure on it, and $\psi$ a finite linear combination of irreducible group characters on $G$. Suppose that $\psi*\psi=\psi$ as elements of $L^1(G,\lambda)$ and that
$\int_G \abs{\psi(x)}\,d\lambda(x) > 1$.
Then
$\int_G \abs{\psi(x)}\,d\lambda(x) \geq 301/300$.
\end{thm}

\begin{rem}
Rider's result is stated for the case where $\lambda(G)=1$. However, if we let $\mu = \lambda(G)^{-1}\lambda$, then $\psi*\psi=\psi$ in $L^1(G,\lambda)$ if and only if $\lambda(G)\psi * \lambda(G)\psi = \lambda(G)\psi$ in $L^1(G,\mu)$. So by rescaling, our formulation reduces to the one given by Rider.
\end{rem}

\begin{qu}
Can we get an improved bound\footnotemark\label{note-in-proof} on the infimum on the left hand side of \eqref{eq:AMZL-gap}, beyond the lower bound $301/300$ provided by Rider's theorem?
\end{qu}
\footnotetext{{\bf Note added in proof:} after this work was accepted for publication, the second author (YC) was able to show that $\AMZL{G}\geq 7/4$ for every finite non-abelian group. Details will appear in a forthcoming paper.}

\begin{rem}
To put this question in context, we note that the smallest explicitly known value of $\AMZL{G}$ for a non-abelian group $G$ is $7/4$ (see Remark~\ref{r:lowest-AMZL-so-far} in the next section). Rider remarks that his estimates are not intended to be best possible, but it seems unlikely that his techniques can get near $11/10$, let alone $7/4$, without substantial new input. Of course, his results concern much more general central idempotents, whereas our concern is with the very particular idempotent described in~\eqref{eq:diagonal}.
\end{rem}

It seems difficult to attack this problem directly using \eqref{eq:AMZL-formula}. One might hope that for groups with two character degrees, one can use \eqref{eq:new-AMZL-formula} to obtain a lower bound on the ZL-amenability constants which is strictly greater than $1$. While we were unable to do this in full generality, we can do better for particular classes of groups; these calculations are the topic of the next section.

Another question raised in \cite[\S1.5]{AzSaSp} is the following:
\begin{quote}
 {\it given a finite non-abelian group $G$, can we get a lower bound on $\AMZL{G}$ in terms of $\max_{\pi\in\widehat{G}} d_\pi$}\/?
\end{quote}
If this were the case, one could obtain further results on (non-)amen\-ability of the centre of $L^1(G)$ for certain profinite groups~$G$.
Unfortunately, as we shall see below (Remark~\ref{r:pain-in-ASS}), there exists a sequence of finite groups $(G_i)$ such that
$\sup_i \max_{\pi\in\hat{G_i}} d_\pi = + \infty$
yet
$\sup_i \AMZL{G_i} = 5$.
Therefore this question has a negative answer.
\end{section}

\begin{section}{ZL-amenability constants of particular groups}
\label{s:AMZL_examples}
Using Theorem \ref{t:amzl-2cd}, we can find the ZL-amenability constants for several well-known families of finite groups.

\begin{subsection}{Dihedral groups}
Let us fix some notation: $D_n$ denotes the \dt{dihedral group of order $2n$}, whose standard presentation~is
\[
D_n=\langle r,t \mid  r^n=t^2=1, tr=r^{-1}t\rangle.
\]
The character table of $D_n$ is well known and can be found in standard sources: for instance, see \cite[pp.~182--183]{JamLie}. We note, nevertheless, that we only need to know the number of linear characters and the cardinalities of the conjugacy classes, both of which can be determined by straightforward {\it ad hoc} arguments that we leave to the reader.
As usual, we must treat the cases of odd and even $n$ separately.

\paragraph{The case of even $n$.}
Suppose $n=2\nu$ for some integer $\nu\geq 2$. Then $D_n$ has four linear characters (so that its derived subgroup has order $\nu$), and all other characters have degree~$2$.
Also, $D_n$ has two conjugacy classes of size $1$ (namely $\{1\}$ and $\{r^\nu\}$), two of size $\nu$ (namely $[t]$ and $[rt]$), and $\nu-1$ of size~$2$ (the remaining rotations, paired~up).
Thus
\[
\sum_{C\in\conj(D_{2\nu})} |C|^2
 = 2\cdot 1^2 + 2\left(\frac{n}{2}\right)^2 + \left(\frac{n}{2}-1\right)\cdot 2^2
 = \frac{1}{2}(n^2 + 4n -4) \,;
\]
and so, by our general formula \eqref{eq:new-AMZL-formula},
\begin{equation}
\label{eq:AMZL-of-dihedral-even}
\begin{aligned}
\AMZL{D_{2\nu}}
 = 1 + 2(2^2-1) \left( 1- \frac{n^2+4n-4}{2n^2}\right)
& = 1 + 6 \frac{n^2-4n+4}{2n^2} \\
& = 1 + 3\left(1 -\frac{2}{n}\right)^2\,.
\end{aligned}
\end{equation}

\paragraph{The case of odd $n$.}
Suppose $n=2\nu+1$ where $\nu$ is an integer $\geq 1$. Then $D_n$ has two linear characters (so that its derived subgroup has order $n$), and all other characters have degree~$2$. Also, its conjugacy classes are as follows: the trivial conjugacy class of the identity; the conjugacy class consisting of all involutions, which has size $n$; and $\nu$ conjugacy classes of size~$2$ (each consisting of a rotation and its inverse).
Thus
\[
\sum_{C\in\conj(D_{2\nu+1})} |C|^2 = 1^2 + n^2 + \frac{n-1}{2}\cdot 2^2 = n^2 +2n -1 \/;
\]
and so, by our general formula \eqref{eq:new-AMZL-formula},
\begin{equation}
\label{eq:AMZL-of-dihedral-odd}
\begin{aligned}
\AMZL{D_{2\nu+1}}
   = 1 + 2(2^2-1)\left( 1- \frac{n^2+2n-1}{2n^2}\right)
 & = 1 + 6 \frac{n^2-2n+1}{2n^2} \\
 & = 1 + 3\left(1-\frac{1}{n}\right)^2 \,.
\end{aligned}
\end{equation}

In fact, when $n$ is odd, $D_n$ fits into a family of more general examples, for which one can simplify \eqref{eq:new-AMZL-formula} even further. These groups are the topic of the next subsection.
\end{subsection}

\begin{subsection}{Frobenius groups with abelian complement and kernel}
\label{ss:frobenius_ACAK}
Frobenius groups admit various characterizations or equivalent definitions. The following one is convenient for our purposes.

\begin{dfn}[cf.~{\cite[Theorem~8.2]{Passman_PGbook}}]
\label{d:frobenius-group}
A finite group $G$ is a \dt{Frobenius group} if it has a finite, proper, non-trivial subgroup $H$ which is \dt{malnormal}, i.e.~which satisfies $H \cap gHg^{-1} = \{e\}$ for all $g \in G \setminus H$. We say that $H$ is a \dt{Frobenius complement} in~$G$.
\end{dfn}

Given a Frobenius complement $H < G$, let $K \defeq \left( G \setminus \bigcup_{g\in G} gHg^{-1} \right) \cup \{e\}$.
Clearly $K$ is a conjugation-invariant susbet of $G$: by a deep result of Frobenius, $K$ is actually a subgroup of $G$, called the \dt{Frobenius kernel} of $G$, and $G$ is the semidirect product $K \rtimes H$.
(See Passman's book, in particular the proof of \cite[Theorem 17.1]{Passman_PGbook}, for further details.)

\begin{rem}
{\it A~priori}, $K$ depends on the particular choice of Frobenius complement~$H$.
However, it turns out that if $G$ has a Frobenius complement $H$ and $K$ is the corresponding Frobenius kernel, then $K$ is equal to the Fitting subgroup of~$G$; moreover, all proper, non-trivial, malnormal subrgoups of $G$ are conjugate in~$G$ (\cite[Cor\-oll\-ary~17.5]{Passman_PGbook}). These highly non-obvious results are sometimes summarized in the slogan ``a finite group can be Frobenius in at most one way''.
\end{rem}

For sake of brevity, we write ``let $G=K\rtimes H$ be Frobenius'' as an abbreviation for ``let $G$ be a finite Frobenius group, with Frobenius complement $H$ and Frobenius kernel $K$.''

\begin{prop}\label{p:frob_facts}
Let $G=K\rtimes H$ be Frobenius. Suppose $H$ is an abelian group of order~$h$, and $K$ is an abelian group of order~$k$.
Then $h$ divides $k-1$. Moreover:
\begin{numlist}
\item\label{li:conjugacy-in-frob}
 $G$ has trivial centre, $(k-1)/h$ conjugacy classes of size $h$, and $h-1$ conjugacy classes of size~$k$.
\item\label{li:characters-of-frob}
 $G$ has exactly $h$ linear characters; the remaining characters each have degree~$h$.
\end{numlist}
\end{prop}

The proposition is an assembly of several standard facts about Frobenius groups. However, as it is difficult to locate a reference that states concisely what we need, we give a proof in Appendix~\ref{app:frobenius_facts}.

\begin{thm}\label{t:AMZL-of-Frobenius}
Let $G$ be a Frobenius group whose complement and kernel are both abelian; let $h$ and $k$ be the orders of the complement and kernel, respectively. Then
\begin{equation}\label{eq:AMZL-of-Frobenius}
\AMZL{G} = 1 + 2\cdot \frac{h^2-1}{h} \left(1 - \frac{h-1}{k}\right)\left(1- \frac{1}{k}\right).
\end{equation}
\end{thm}

\begin{proof}
By Proposition~\ref{p:frob_facts},
\[ \sum_{C\in\conj(G)} |C|^2 = 1 +
\frac{k-1}{h}\ h^2 + (h-1)\ k^2 = 1 + h(k-1) + (h-1)k^2 \,;
\]
and substituting the remaining information from Proposition~\ref{p:frob_facts} into the general formula \eqref{eq:new-AMZL-formula} yields
\[ \begin{aligned}
\frac{\AMZL{G} -1}{2}
& =  (h^2-1) \left( 1- \frac{ 1 + h(k-1) + (h-1)k^2}{hk^2} \right) \\
& =  \frac{h^2-1}{h} \cdot \frac{ -1 - hk+h + k^2}{k^2}  \\
& =  \frac{h^2-1}{h} \left( 1-\frac{h}{k} + \frac{h-1}{k^2}\right)\,;
\end{aligned} \]
factorizing and rearranging this gives the formula \eqref{eq:AMZL-of-Frobenius}, as required.
\end{proof}

\begin{eg}[Dihedral groups of odd order, revisited]
\label{eg:AMZL-of-dihedral-odd}
Let $n$ be an odd integer with $n\geq 3$. Using the standard presentation of $D_n$ as given earlier, we see that the subgroup generated by the `reflection' $t$ is malnormal, while the Frobenius kernel turns out to be the subgroup generated by the `rotation'~$r$. Putting $h=2$ and $k=n$ in \eqref{eq:AMZL-of-Frobenius} gives
\[ \AMZL{D_n} = 1 + 3 \left(1- \frac{1}{n}\right)^2 \,, \]
just as we had before.
\end{eg}

\begin{eg}[Affine groups of finite fields]
\label{eg:AMZL-of-affine}
Let $\bbF_q$ be a finite field of order $q$, where $q$ is a prime power $\geq 3$.
The \dt{affine group of $\bbF_q$}, which we shall denote by $\Aff(\bbF_q)$, is the set
\[ \left\{ \affine{a}{b} \st a \in\bbF_q^\times, b\in \bbF_q \right\} \]
equipped with the group structure it inherits from the usual matrix product and inversion.
It is a metabelian group; more precisely, it is isomorphic to the semidirect product $\bbF_q \rtimes \bbF_q^\times$\/.

It is straightforward to check that the subgroup of $\Aff(\bbF_q)$ corresponding to the multiplicative group of $\bbF_q$ is a proper, non-trivial, malnormal subgroup; the Frobenius kernel turns out to be the normal subgroup of $\Aff(\bbF_q)$ corresponding to the additive group of~$\bbF_q$.
Both are abelian, so we can apply Theorem~\ref{t:AMZL-of-Frobenius}, which yields
\[ \begin{aligned}
\AMZL{\Aff(\bbF_q)}
 &  =  1 + 2\cdot\frac{(q-1)^2-1}{q-1}\left(1 - \frac{q-2}{q}\right) \left( 1 -\frac{1}{q}\right) \\
 &  =  1 + 2\cdot\frac{2q-q^2}{q-1}\cdot\frac{2}{q}\cdot \frac{q-1}{q} \\
 &  =  1 + 4\cdot \frac{q-2}{q} \\
 &  =  5 - \frac{8}{q} \,.
\end{aligned} \]
\end{eg}

\begin{rem}
One can also compute $\AMZL{\Aff(\bbF_q)}$ more directly from the character table of $\Aff(q)$, which is simple and well-known, and can be found in standard sources. In fact, the exact computation for these examples, which arose in other work of the present authors related to \cite{Stegmeir}, provided some of the motivation for Theorem~\ref{t:amzl-2cd}.
\end{rem}

\begin{rem}\label{r:pain-in-ASS}
For all odd primes $p$, $2\leq \AM(\Aff(p)) \leq 5$, while $\Aff(p)$ has an irreducible representation of dimension $p-1$. This shows that the amenability constant of $\Zl^1(G)$ cannot be bounded from below by an increasing function of $\max\{ d_\chi : \chi \in \Irr(G)\}$. (For context, see the remarks made at the end of Section~\ref{s:AMZL-formulas}.)
\end{rem}

\begin{eg}[$a^2x+b$ groups]
\label{eg:AMZL-of-subgroup}
Let $q$ be an odd prime power $\geq 5$, and let $d=(q-1)/2$.
Consider the following subgroup of $\Aff(\bbF_q)$, sometimes referred to as the
``$a^2x+b$ group over~$\bbF_q$''\/:
\[ G_q \defeq \left\{ \left( \begin{matrix} a^2 & b \\ 0 & 1 \end{matrix} \right)  \st a\in \bbF_q^\times, b\in \bbF_q \right\}. \]

Recalling that $\bbF_q^\times$ is cyclic, pick a generator $z$, and let $H$ be the subgroup of $\Aff(\bbF_q)$ generated by $\left(\begin{matrix} z^2 & 0 \\ 0 & 1 \end{matrix} \right)$.
One can check that $H$~is malnormal, and so $G_q$ is Frobenius; the Frobenius kernel $K$ turns out to be the normal subgroup corresponding to the additive group of $\bbF_q$. So both $K$ and $H$ are abelian; the former has order $q$ while the latter has order $d$,
so using \eqref{eq:AMZL-of-Frobenius} we get
\[
\begin{aligned}
\AMZL{G_q}
& = 1 + 2\cdot\frac{d^2-1}{d}\left(1- \frac{d-1}{q}\right)\left(1-\frac{1}{q}\right) \\
& = 1 + 2\cdot\frac{q^2-2q-3}{2(q-1)}\ \frac{q+3}{2q} \ \frac{q-1}{q}
\end{aligned} \]
which simplifies to
\begin{equation}\label{eq:AMZL-of-subgroup}
\AMZL{G_q} = 1 + \frac{q+1}{2}\ \left( 1 - \frac{9}{q^2} \right) \\
\end{equation}
\end{eg}

As a consistency check: when $q=5$, Equation~\eqref{eq:AMZL-of-subgroup} gives $\AMZL{G_5} = 73/25$. On the other hand, it is straightforward to check that $G_5$ is isomorphic to the dihedral group of order $10$, and
using our earlier formulas we have $\AMZL{D_5} = 73/25$.

\begin{rem}
Even though $G_q$ is a (index $2$) subgroup of $\Aff(\bbF_q)$, it may have a larger ZL-amenability constant. Indeed, it is clear from the formulas obtained in Examples \ref{eg:AMZL-of-affine} and \ref{eg:AMZL-of-subgroup} that
\[ \lim_{q\to\infty} \AMZL{\Aff_q} = 5 \quad\text{while}\quad \lim_{q\to\infty} q^{-1}\AMZL{G_q} = \frac{1}{2} \]
\end{rem}

Example~\ref{eg:AMZL-of-subgroup} shows that within the class of groups with two character degrees, we can obtain arbitrarily large ZL-amenability constants. It is natural to ask how small such constants can be. For Frobenius groups with abelian complement and kernel, we can obtain a complete answer.

\begin{thm}\label{t:lower-bound-AMZL-Frob-AKAC}
Let $G$ be a Frobenius group with abelian complement and kernel. Then\hfil\newline
$\AMZL{G} \geq 7/3$,
with equality if and only if $G$ is the dihedral group of order $6$.
\end{thm}

\begin{proof}
Let $h$ be the order of the Frobenius complement of $G$, and $k$ the order of its Frobenius kernel. Note that $G$ is isomorphic to $D_3$ if and only if $h=2$ and $k=3$.

To reduce notational clutter, let $F(k,h)$ denote $\frac{1}{2}(\AMZL{G}-1)$. By Theorem~\ref{t:AMZL-of-Frobenius},
\begin{equation}\label{eq:DAFFY}
F(k,h) =
\frac{h^2-1}{h} \left(1 - \frac{h-1}{k}\right)\left(1- \frac{1}{k}\right)
\tag{$\clubsuit$}
\end{equation}
and it suffices to prove that $F(k,h)\geq 2/3$ with equality if and only if $(h,k)=(2,3)$ (subject to $h$ and $k$ arising from a Frobenius group of the specified form).

Note that for fixed $h$, $F(\cdot, h)$ is a strictly increasing function.
As observed above, $h$ divides $k-1$, so in particular $k\geq h+1$; hence, $F(k,h) \geq F(h+1,h)$, with equality if and only if $h=k-1$.
Direct calculation gives
\[ \begin{aligned}
F(h+1,h)
  = \frac{h^2-1}{h} \cdot \frac{2}{h+1} \cdot \frac{h}{h+1}
  = \frac{2(h-1)}{h+1} = 2\left(1-\frac{2}{h+1}\right),
\end{aligned} \]
and so $F(h+1,h) \geq 2/3$, with equality if and only if $h=2$. This completes the proof.
\end{proof}

If we consider more general groups with two character degrees, then there is an infinite family of such groups whose ZL-amenability constants are less than~$2$. This will be seen in the next and final subsection of the paper.
\end{subsection}

\begin{subsection}{Extra-special $p$-groups.}

\begin{dfn}
\label{d:extraspecial}
Fix a prime $p$. A finite group $G$ is \dt{$p$-extraspecial} if it has order $p^{2n+1}$ for some integer $n$ and has the following properties:
\begin{numlist}
\item The centre $Z(G)$ and the derived subgroup $\der{G}$ both have order $p$.
\item The quotient $G/Z(G)$ is abelian, and each non-identity element in the quotient has order~$p$.
\end{numlist}
\end{dfn}

Such groups do exist (for instance, the dihedral group of order $8$ is $2$-extraspecial), and their character tables and conjugacy classes turn out to be uniquely determined by these conditions.
In particular, each non-linear irreducible group character of $G$ is supported on $Z(G)$ and has degree~$p^n$.
This follows from, e.g. \cite[Chapter 5, Theorem 5.5]{Gor_FGbook_ed2}, as pointed out to the second author by D.~F. Holt. Alternatively, a short argument using some basic character theory is described by I.~M. Isaacs in the appendix to \cite{Dia_threads}.

\begin{eg}\label{eg:AMZL-of-extraspecial}
Let $G$ be an extraspecial group of order $p^{2n+1}$, where $p$ is a prime. 
We know every non-linear character has degree $p^n$, and we know $\der{G}$ has order~$p$. To apply Theorem~\ref{t:amzl-2cd}, we need also to know the sizes of the conjugacy classes. By some elementary group theory (see e.g.~the appendix of \cite{Dia_threads}), the conjugacy classes of $G$ are either elements of the centre or the non-trivial cosets of the derived subgroup. Thus there are $p$ conjugacy classes of size $1$ and $p^{2n}-1$ conjugacy classes of size~$p$, and no others.
Therefore
\[
\sum_{C\in\conj(G)} |C|^2 = p\cdot 1^2 + (p^{2n}-1)\cdot p^2 = p^{2n+2}-p^2+p
\]
and so Theorem~\ref{t:amzl-2cd} gives
\begin{equation}
\label{eq:AMZL-of-extraspecial}
\begin{aligned}
\AMZL{G}
& = 1 + 2(p^{2n}-1)\left( 1 - \frac{p^{2n+2}-p^2+p}{ p^{2n+2} } \right) \\
& = 1 + 2\left(1 - \frac{1}{p^{2n}} \right)\ \left(1 - \frac{1}{p} \right)  \\
\end{aligned}
\end{equation}
\end{eg}

\begin{rem}
If $G$ is an extraspecial group of order $2^{2n+1}$, then
$\AMZL{G} = 2 - 2^{-2n}$\/. Thus we have an infinite family of finite groups $G$ for which $1 < \AMZL{G} < 2$.
\end{rem}

\begin{rem}\label{r:lowest-AMZL-so-far}
Within the class of extra-special $p$-groups, the ZL-amenability constant is minimized when we take $p=2$ and $n=1$. This example is nothing but the dihedral group of order $8$, whose amenability constant is $7/4$.
This is the smallest ZL-amenability constant we have found for any non-abelian group.
\end{rem}

\begin{qu}
Does there exist a non-abelian, finite group whose ZL-amenability constant is $<7/4$? Can such a group have two character degrees?
\end{qu}

\end{subsection}

\subsection{Summary information}
We summarize the findings of this section in Figure~\ref{fig:table}.
\begin{figure}[hpt]
\begin{tabular}{l|c|c|c|c|c|r}
   Ref. &		$G$
	&     $|G|$	& $|\fL|$ & c.d.
	& $\AMZL{G}-1$  & min. \\

\hline

Ex.~\ref{eg:AMZL-of-affine} & $\Aff(\bbF_q)$,
	& $q(q-1)$   &   $q-1$	& $q-1$
	& $4(1-2q^{-1})$	&	$4/3$ \\
	& $q\geq 3$
	&	& 	& 	& \\

Ex.~\ref{eg:AMZL-of-subgroup} & $ax^2+b$ of $\bbF_q$,
	& ${\displaystyle\frac{q(q-1)}{2}}$ & $\frac{1}{2}(q-1)$ & $q-1$
	& $\frac{1}{2}(q+1) (1 - 9q^{-2})$  &	$48/25$ \\
	& $q\geq 5$
	&	& 	& 	& \\

Eq.~\eqref{eq:AMZL-of-dihedral-odd} & $D_n$,
	& $2n$ & $2$ & $2$
	& $3(1-n^{-1})^2$	&	$4/3$ \\
	& $n$ odd $\geq 3$
	&	& 	& 	& \\

Eq.~\eqref{eq:AMZL-of-dihedral-even} & $D_n$,
	& $2n$ & $4$ & $2$
	& $3(1-(2n)^{-1})^2$	&	$3/4$ \\
	& $n$ even $\geq 4$
	&	& 	& 	& \\

Ex.~\ref{eg:AMZL-of-extraspecial} & $p$-extraspecial
	& $p^{2n+1}$ & $p^{2n}$ & $p^n$
	& $2(1- p^{-2n}) (1-p^{-1})$  &	$3/4$
\end{tabular}
\caption{Summary table for some groups with two character degrees}
\label{fig:table}
\end{figure}
\begin{itemize}
\item ``Ref.'' gives the number of the relevant theorem, example or equation.
\item $\fL$ is the set of linear characters.
\item ``c.d.'' stands for the character degree of the non-linear characters.
\item ``min.'' denotes the minimum value of $\AMZL{G}-1$ within the specified family of groups.
\end{itemize}

\end{section}

\appendix
\begin{section}{Properties of Frobenius groups}
\label{app:frobenius_facts}
In this appendix we collect the facts about Frobenius groups that are needed to prove Proposition~\ref{p:frob_facts}. Since we are interested only in the special case where both complement and kernel are abelian, it will sometimes be easier to give short proofs, than to cite general results and then specialize. On the other hand, in some places we shall merely give appropriate references to the literature.

Throughout, $G=K\rtimes H$ is Frobenius; $k$ and $h$ denote the orders of $K$ and $H$ respectively.
By considering the permutation action of $G$ on cosets of $H$, it follows easily from the malnormal property that $h$ divides~$k-1$. (This does not need Frobenius's result that $K$ is a group.)

\begin{proof}[Proof of Proposition~\ref{p:frob_facts}~\ref{li:conjugacy-in-frob}]
Throughout this proof $x^G$ denotes the conjugacy class of $x$ in $G$, and $C_G(x)$ denotes the centralizer of $x$ in~$G$.
We repeatedly use the fact that each element of $G$ can be written either as $xb$ where $x\in K$ and $b\in H$, or as $ay$ where $a\in H$ and $y\in K$. (This is immediate from the decomposition of $G$ as a semidirect product of $H$ and~$K$.)

The first step is to identify the conjugacy classes of $K$ inside $G$.
Let  $x\in K \setminus\{e\}$. By the malnormal property, $C_G(x)\cap H =\{e\}$.
Since $K$ is abelian and $G=HK$, it follows that $C_G(x)=K$. Therefore $|x^G| = |G| / |K| = h$.

For the second step, recall that by definition, $G\setminus K = \left( \bigcup_{g\in G} gHg^{-1}\right) \setminus\{e\}$.
Since $H$ is abelian and $G=KH$, we obtain $G\setminus K = \left( \bigcup_{x\in K} xHx^{-1}\right) \setminus\{e\}$.
Now, by malnormality of $H$ inside $G$, and the fact that $H\cap K=\{e\}$, we see that
\[ xHx^{-1}\cap yHy^{-1}= \{e\} \qquad\text{ whenever $x,y\in K$ and $x\neq y$.} \]
Thus, the function $H\setminus\{e\} \to \conj(G)$, $a\mapsto a^G$, is injective, and $|a^G| = k$ for each $a\in H\setminus\{e\}$. This gives us the required partition of $G\setminus K$ into $h-1$ disjoint conjugacy classes, each of size~$k$.
\end{proof}

To prove the second part of Proposition~\ref{p:frob_facts}, we appeal to some general results on the character theory of Frobenius groups.

\begin{propstar}
Let $G=K\rtimes H$ be Frobenius. The following belong to $\Irr(G)$.
\begin{itemize}
\item the characters arising by composing irreducible characters of $H$ with the quotient map $G \to G/K \cong H$;
\item the characters arising by inducing an irreducible character of $K$ up to $G$.
\end{itemize}
Moreover, every irreducible character of $G$ arises in this way.
\end{propstar}

\begin{proof}
See, for example, \cite[Theorem~6.34]{Isaacs_CTbook}.
\end{proof}

In the special case where $H$ and $K$ are abelian, it follows immediately that $G$ has two character degrees. The irreducible characters of $G$ that arise by inducing irreducible characters from $K$ all have degree $|G:K|=h$; the remaining characters are all linear, arising from the irreducible characters of~$H$, and there are precisely $|\widehat{H}| = |H| = h$ of them. This completes the proof of Proposition~\ref{p:frob_facts}~\ref{li:characters-of-frob}.
\end{section}

\subsection*{Acknowledgements}
The first author (MA) was supported by a Dean's Scholarship from the University of Saskatchewan. The second author (YC) was partially supported by NSERC Discovery Grant 402153--2011 and the third author (ES) by NSERC Discovery Grant 366066--2009.
The authors thank the referee for some useful corrections and suggestions.


\vfill\eject

\vfill
\noindent Mahmood Alaghmandan, Yemon Choi, and Ebrahim Samei

\noindent
Department of Mathematics and Statistics\\
McLean Hall, University of Saskatchewan\\
Saskatoon (SK), Canada S7N 5E6

\smallskip\noindent
E-mails: \texttt{mahmood.a@usask.ca}, \texttt{choi@math.usask.ca}, \texttt{samei@math.usask.ca} 

\end{document}